\documentclass[11pt,reqno]{amsart}
\usepackage[T1]{fontenc}
\usepackage{lmodern}
\usepackage{amsmath,amssymb,amsthm,mathtools,mathrsfs}
\usepackage{microtype}
\usepackage{enumitem}
\usepackage[colorlinks=true,linkcolor=black,citecolor=black,urlcolor=blue]{hyperref}
\usepackage[margin=1.09in]{geometry}
\allowdisplaybreaks[2]
\setlist[enumerate]{leftmargin=*,itemsep=2pt,topsep=3pt}
\newtheorem{theorem}{Theorem}[section]
\newtheorem{proposition}[theorem]{Proposition}
\newtheorem{corollary}[theorem]{Corollary}
\newtheorem{lemma}[theorem]{Lemma}
\theoremstyle{definition}
\newtheorem{definition}[theorem]{Definition}
\newtheorem{example}[theorem]{Example}
\theoremstyle{remark}
\newtheorem{remark}[theorem]{Remark}
\newcommand{\NN}{\mathbb{N}}
\newcommand{\CC}{\mathcal{C}}
\newcommand{\PP}{\mathcal{P}}

\newcommand{\Ann}{\operatorname{Ann}}
\newcommand{\Frac}{\operatorname{Frac}}
\newcommand{\inv}{\mathrm{inv}}

\title[Radical well-mixed difference ideals]{The ascending chain condition for radical well-mixed difference ideals}
\author{Wei Li}
\address{State Key Laboratory of Mathematical Sciences, Academy of Mathematics and Systems Science, Chinese Academy of Sciences, Beijing, China}
\email{liwei@mmrc.iss.ac.cn}
\author{Jie Wang}
\address{State Key Laboratory of Mathematical Sciences, Academy of Mathematics and Systems Science, Chinese Academy of Sciences, Beijing, China}
\email{wangjie212@amss.ac.cn}
\date{October 21, 2026}
\begin{document}
\begin{abstract}
We establish a finite-basis theorem for radical well-mixed difference ideals in a difference polynomial ring with one endomorphism over an arbitrary difference field. 
The result extends to every finitely difference-generated difference algebra over such a field and is equivalent to the ascending chain condition for radical well-mixed
difference ideals, hence resolving a long-standing open problem in difference algebra.
The proof proceeds in two stages. Over an inversive difference coefficient field, we establish a uniform inverse-transform bound relating the perfect closure of a finite system to its radical well-mixed closure. This uses Noetherian
induction in a fixed truncated polynomial ring, together
with generic freeness and going-down for flat ring
homomorphisms. We then prove that the radical well-mixed closure of a system over a difference field $K$ is the contraction of its radical well-mixed closure over the inversive closure of $K$. This contraction result follows from successive prime lifting along flat extensions
and removes the inversiveness assumption without requiring
the coefficient-field extension to be finitely generated.
\end{abstract}
\maketitle

\section{Introduction}

The Ritt--Raudenbush basis theorem states that the perfect difference ideals of a difference polynomial ring in finitely many difference indeterminates satisfy the ascending chain condition; equivalently, every perfect difference ideal admits a finite perfect basis \cite{RittRaudenbush,WibmerNotes}. Here finite generation refers to the perfect closure, not to ordinary ideal generation. Perfect ideals are radical, well-mixed, and reflexive. Removing reflexivity leads to a different finiteness question.

In his seminal work on the elementary theory of the Frobenius
automorphisms, Hrushovski studied radical well-mixed difference
ideals and proved a finite-basis theorem for finitely generated
well-mixed difference algebras under a boundedness condition
\cite[Lemma~4.35]{Hrushovski}. Levin showed that the analogous assertion for arbitrary well-mixed ideals fails, even in one difference indeterminate \cite{Levin}. Positive results were subsequently established for monomial and binomial difference ideals \cite{WangMonomial,WangBinomial}. Wibmer formulated the unrestricted radical well-mixed chain condition as Conjecture B and related it to finiteness of minimal prime difference ideals; he also proved the relevant finiteness assertion for finitely difference-generated difference Hopf algebras \cite{WibmerGroups}. Related results for partial difference algebraic groups are discussed in \cite{McGrath}.

A \emph{difference field} is a field $K$ with a (necessarily injective) endomorphism $\sigma_K$. It is \emph{inversive} if $\sigma_K$ is surjective. Throughout the paper, we work with a single distinguished
endomorphism. We denote the difference polynomial ring
over $K$ in the difference variables $y_1,\ldots,y_n$ by
\[
 R=K\{y_1,\ldots,y_n\}
  =K[y_{i,j}:1\leq i\leq n,\ j\in\NN],
\]
where $y_{i,j}=\sigma^j(y_i)$, and the endomorphism $\sigma$ of $R$ is determined by
\(
 \sigma|_K=\sigma_K,\) and 
 \(
 \sigma(y_{i,j})=y_{i,j+1}.
\)
For a subset $U$ of a difference ring $A$, we write $\CC_A(U)$ for the radical well-mixed difference ideal generated by $U$, and $\PP_A(U)$ for its perfect closure. These closures will be defined precisely in Section~\ref{sec:closures}.

Our main result is as follows.

\begin{theorem}[Radical well-mixed basis theorem]\label{thm:main}
Let $(K,\sigma_K)$ be any difference field and let $n\ge0$. Every radical well-mixed difference ideal $I$ of $K\{y_1,\ldots,y_n\}$ has a finite radical well-mixed basis: there is a finite set $F\subseteq I$ such that
\[
 I=\CC_{K\{y_1,\ldots,y_n\}}(F).
\]
Consequently, the radical well-mixed difference ideals of $K\{y_1,\ldots,y_n\}$ satisfy the ascending chain condition.
\end{theorem}

The corresponding quotient-algebra statement is worth recording separately. It does not require the endomorphism on the quotient to be injective.

\begin{corollary}\label{cor:intro}
Let $(K,\sigma_K)$ be any difference field, and let $A$ be a  finitely difference-generated difference $K$-algebra.  Then every radical well-mixed difference ideal $J$ of $A$ has a finite radical well-mixed basis. Equivalently, every ascending chain of radical well-mixed difference ideals of $A$ stabilizes.
\end{corollary}

The proof also yields a uniform comparison of two closure operations. For any ordinary ideal $J$ and $m\ge0$, $\sigma^{-m}(J)$ always means the \emph{inverse image} under $\sigma^m$, not application of an inverse endomorphism.

\begin{theorem}[Uniform inverse-transform bound]\label{thm:uniform-intro}
Let $(K,\sigma_K)$ be an arbitrary difference field, let $R=K\{y_1,\ldots,y_n\}$, and let $F\subset R$ be finite. There is $N=N(F)\ge0$ such that
\begin{equation}\label{eq:uniform-main}
 \PP_R(F)=\sigma^{-N}\!\bigl(\CC_R(F)\bigr).
\end{equation}
Equivalently, $\sigma^N(\PP_R(F))\subseteq\CC_R(F)$.
\end{theorem}

The finiteness assumption on $F$ can be removed using the
Ritt--Raudenbush basis theorem; see
Corollary~\ref{cor:uniform-arbitrary}.
The key point is that the exponent $N$ is uniform over
all elements of the perfect closure of the fixed system $F$.
The corresponding elementwise statement, in which the
exponent may depend on the element, is elementary but
does not suffice for the finite-basis argument.
No effective bound for $N$ in terms of degrees and orders
is asserted.

We first prove the uniform comparison over inversive fields, using only ordinary Noetherianity of one truncated polynomial ring. The crucial geometric input is a localized flat-lifting lemma for prime difference ideals. A finite-type transfer principle then allows us to combine the uniform bound with the classical perfect basis theorem. The remaining coefficient-field hypothesis is removed by a separate descent argument. If $K^{\inv}$ denotes the inversive closure of $K$, we prove the contraction identity
\begin{equation}\label{eq:descent-intro}
 \CC_{K^{\inv}\{\mathbf y\}}(U)\cap K\{\mathbf y\}
  =\CC_{K\{\mathbf y\}}(U)\qquad(U\subseteq K\{\mathbf y\}).
\end{equation}
The important point is that no finite-generation assumption on $K^{\inv}/K$ is needed. Unlike extension of ordinary prime ideals, the appropriate extension of a prime difference ideal must be selected using the difference operator. We do so one coefficient field at a time, applying flat going-down.

Section~\ref{sec:closures} establishes the closure operations. Section~\ref{sec:lifting} develops the localized lifting lemma. The inverse-transform bound over an inversive field is proved in Section~\ref{sec:uniform}. Section~\ref{sec:transfer} gives the finite-type transfer argument, and Section~\ref{sec:inversive-main} completes the inversive case. Section~\ref{sec:descent} establishes coefficient descent and proves Theorems~\ref{thm:main} and~\ref{thm:uniform-intro} in their stated generality. We conclude with examples and limitations.

\section{Difference ideals and their closures}\label{sec:closures}
In this section, we recall basic notions from difference algebra
\cite{Cohn,Levin} and establish properties of radical well-mixed
and perfect closures that will be used later.

All rings are commutative with identity. A difference ring $(A,\sigma)$ is a ring with an endomorphism $\sigma:A\to A$; the endomorphism need not be injective. A \emph{difference ideal} is an ordinary ideal $I$ for which $\sigma(I)\subseteq I$.

\begin{definition}\label{def:ideals}
A difference ideal $I$ is \emph{radical} if $a^r\in I$ with $r\ge1$ implies $a\in I$; it is \emph{well-mixed} if $ab\in I$ implies $a\sigma(b)\in I$; it is \emph{reflexive} if $\sigma(a)\in I$ implies $a\in I$. It is \emph{perfect} if
\[
 \sigma^{i_1}(a)\cdots\sigma^{i_s}(a)\in I
 \quad\Longrightarrow\quad a\in I
 \qquad (s\ge1,\ i_1,\ldots,i_s\in\NN).
\]
An ordinary prime ideal which is a difference ideal is called a \emph{prime difference ideal}; reflexivity
is not required.
\end{definition}

For $U\subseteq A$, its difference ideal is $[U]_A=(\sigma^j(u):u\in U,\ j\ge0)_A$. Arbitrary intersections of radical well-mixed ideals, or of perfect ideals, remain in the respective class. We use $\CC_A(U)$ and $\PP_A(U)$ for the least such ideals containing $U$. We write $\CC_A(U,b)=\CC_A(U\cup\{b\})$ and similarly for $\PP$. The notation $\CC_A(0)$ means $\CC_A(\varnothing)$; it need not be the zero ideal. We suppress subscripts when there is no ambiguity.

\begin{lemma}[Perfectness criterion]\label{lem:perfect-criterion}
A difference ideal is perfect if and only if it is radical, well-mixed, and reflexive.
\end{lemma}
\begin{proof}
Suppose $I$ is perfect. Taking all transform indices equal to zero gives radicality, and taking one transform of index one gives reflexivity. If $ab\in I$, set $u=a\sigma(b)$. Since
\[
 u\sigma(u)=a\sigma^2(b)\sigma(ab)\in I,
\]
perfectness implies $u\in I$. Thus $I$ is well-mixed.

Conversely, suppose the three conditions hold and a product of transforms of $a$ belongs to $I$. By well-mixedness, any factor in a product belonging to $I$ may be replaced by its transform. Raising all transform indices to their maximum $m$ gives $\sigma^m(a)^s\in I$. Radicality yields $\sigma^m(a)\in I$; reflexivity applied $m$ times yields $a\in I$.
\end{proof}

\begin{lemma}[Quotients and inverse images]\label{lem:functorial}
Let $f:(A,\sigma)\to(B,\tau)$ be a difference-ring homomorphism. The inverse image of a radical well-mixed (respectively perfect) ideal of $B$ is an ideal of the same class in $A$. For every $U\subseteq A$,
\[
 f(\CC_A(U))\subseteq\CC_B(f(U)),\qquad
 f(\PP_A(U))\subseteq\PP_B(f(U)).
\]
If $f$ is surjective with kernel $J$, then
\begin{equation}\label{eq:quotient}
 f^{-1}\bigl(\CC_B(f(U))\bigr)=\CC_A(U\cup J),\qquad
 f^{-1}\bigl(\PP_B(f(U))\bigr)=\PP_A(U\cup J).
\end{equation}
\end{lemma}
\begin{proof}
The inverse-image assertions follow by applying $f$ to the defining implications, including the well-mixed implication $ab\in I\Rightarrow a\sigma(b)\in I$. Minimality gives the two inclusions. For a surjection, the usual correspondence between ideals of $A$ containing $J$ and ideals of $B$ preserves difference stability, radicality, well-mixedness, and reflexivity in both directions. Lemma~\ref{lem:perfect-criterion} deals with perfectness. Formula~\eqref{eq:quotient} follows.
\end{proof}

\begin{lemma}[Finitarity]\label{lem:finitarity}
For every subset $U$ of a difference ring $A$,
\begin{equation}\label{eq:finitary}
 \CC_A(U)=\bigcup_{\substack{F\subseteq U\\F\text{ finite}}}\CC_A(F),\qquad
 \PP_A(U)=\bigcup_{\substack{F\subseteq U\\F\text{ finite}}}\PP_A(F).
\end{equation}
In particular, the union of an ascending chain of radical well-mixed ideals is radical well-mixed, and the corresponding assertion holds for perfect ideals.
\end{lemma}
\begin{proof}
Let $(I_\lambda)_\lambda$ be a nonempty family of difference
ideals directed by inclusion, and put $I=\bigcup_\lambda I_\lambda$.
Then $I$ is a difference ideal: any two of its elements lie
in a common $I_\lambda$, and each $I_\lambda$ is an ideal
stable under $\sigma$.
Moreover, if every $I_\lambda$ is radical and well-mixed,
then so is $I$; if every $I_\lambda$ is perfect, then so is $I$.
Indeed, any power or product whose membership in $I$
occurs in these defining conditions already belongs to
some $I_\lambda$, where the required implication holds.

The families on the right-hand side of \eqref{eq:finitary}
are directed, since the closures associated with finite
sets $F,G\subseteq U$ are both contained in the corresponding
closure of $F\cup G$.
Thus each union belongs to the appropriate class and
contains $U$. By minimality, it contains the corresponding
closure of $U$. The reverse inclusion follows from
monotonicity, proving both equalities in (\ref{eq:finitary}).

The assertion about ascending chains follows  from the directed-union argument above.
\end{proof}

The following lemma is due to Hrushovski
\cite[Lemma~2.10]{Hrushovski}, which provides prime representations for radical well-mixed difference ideals. We include a short proof
for completeness.
\begin{lemma}\label{lem:prime-representation}
Every radical well-mixed difference ideal of a difference
ring $A$ is an intersection of prime difference ideals.
Consequently, for every subset $U\subseteq A$,
\begin{equation}\label{eq:prime-representation}
 \CC_A(U)=
 \bigcap_{\substack{Q\supseteq U\\
                    Q\text{ prime difference ideal of }A}}Q,
\end{equation}
where an empty intersection is interpreted as $A$.
\end{lemma}

\begin{proof}
Let $I$ be a radical well-mixed difference ideal of $A$,
and set $D=A/I$. Then $D$ is reduced and well-mixed (i.e., $ab=0$ implies that $a\sigma(b)=0$).
We show that every ordinary minimal prime $\mathfrak q$
of $D$ is closed under $\sigma$.

The prime ideals of $D_{\mathfrak q}$ correspond to the
prime ideals of $D$ contained in $\mathfrak q$.
By minimality of $\mathfrak q$, the only prime ideal of
$D_{\mathfrak q}$ is therefore $\mathfrak qD_{\mathfrak q}$.
Since $D_{\mathfrak q}$ is reduced, its nilradical is zero,
so $\mathfrak qD_{\mathfrak q}=0$.
Thus, for each $x\in\mathfrak q$, we have $x/1=0$ in
$D_{\mathfrak q}$, and hence $sx=0$ for some
$s\in D\setminus\mathfrak q$.
Well-mixedness gives $s\sigma(x)=0$.
As $\mathfrak q$ is prime and $s\notin\mathfrak q$,
we obtain $\sigma(x)\in\mathfrak q$.

Thus every ordinary minimal prime of $D$ is a prime
difference ideal. Their intersection is the nilradical
of $D$, which is zero. Taking inverse images under
$A\to D$ expresses $I$ as an intersection of prime
difference ideals containing $I$.
Since every prime difference ideal is radical and well-mixed, a prime difference ideal contains $U$
if and only if it contains $\CC_A(U)$.
Applying the preceding result to $I=\CC_A(U)$ gives
\eqref{eq:prime-representation}.
\end{proof}

\begin{lemma}[Reflexive closure]\label{lem:reflexive-closure}
If $I$ is a radical well-mixed difference ideal, then
\begin{equation}\label{eq:reflexive}
 \PP_A(I)=\bigcup_{m\ge0}\sigma^{-m}(I).
\end{equation}
Consequently, for every $U\subseteq A$,
\begin{equation}\label{eq:perfect-union}
 \PP_A(U)=\bigcup_{m\ge0}\sigma^{-m}\bigl(\CC_A(U)\bigr).
\end{equation}
\end{lemma}
\begin{proof}
The inverse images $\sigma^{-m}(I)$ are radical well-mixed by Lemma~\ref{lem:functorial}. They form an ascending chain because $\sigma(I)\subseteq I$. Their union is radical well-mixed by Lemma~\ref{lem:finitarity}, and it is reflexive since $\sigma(x)\in \sigma^{-m}(I)$  implies $x\in \sigma^{-(m+1)}(I)$. It is thus perfect. Every perfect ideal containing $I$ contains all these inverse images, proving \eqref{eq:reflexive}. Apply it to $I=\CC_A(U)$ to obtain \eqref{eq:perfect-union}.
\end{proof}

\begin{lemma}[Products and intersections]\label{lem:product-intersection}
For subsets $U,V\subseteq A$, write $UV=\{uv:u\in U,v\in V\}$. Then
\begin{equation}\label{eq:intersection}
 \CC_A(U)\cap\CC_A(V)=\CC_A(UV),\qquad
 \PP_A(U)\cap\PP_A(V)=\PP_A(UV).
\end{equation}
The same assertions hold with the ordinary product when $U,V$ are ideals.
\end{lemma}
\begin{proof}
An ordinary prime ideal contains $UV$ if and only if it contains $U$ or contains $V$. Equation~\eqref{eq:prime-representation} proves the first equality. For the second, use \eqref{eq:perfect-union}; since the inverse-image chains are increasing, one may synchronize their indices and obtain
\[
 \PP_A(U)\cap\PP_A(V)
 =\bigcup_{m\ge0}\sigma^{-m}\bigl(\CC_A(U)\cap\CC_A(V)\bigr)
 =\PP_A(UV).
\]
Passing to ordinary ideal spans changes neither closure.
\end{proof}

\begin{remark}\label{rem:prelim-scope}
All results of this section hold for arbitrary difference rings, without an injectivity or Noetherianity hypothesis. In particular, the finitary statements do not presuppose the ascending chain condition to be proved.
\end{remark}

We conclude this section by recalling the classical
Ritt--Raudenbush basis theorem
\cite{RittRaudenbush}; see also
\cite[Theorem~2.5.11]{Levin}.

\begin{theorem}[Ritt--Raudenbush Basis Theorem]\label{thm:Ritt--Raudenbush}
Let $(K,\sigma_K)$ be an arbitrary difference field and
let $R=K\{y_1,\ldots,y_n\}$.
Every perfect difference ideal $P$ of $R$ has a finite
perfect basis. That is, there exists a finite set $E\subseteq P$
such that
\(
 P=\PP_R(E).
\)
\end{theorem}

\section{Localized flat lifting}\label{sec:lifting}

We use standard properties of flat ring homomorphisms: stability under base change and localization, and going-down for chains of ordinary prime ideals \cite{Eisenbud1995}. The localizations below are ordinary localizations and are not assumed to carry a difference endomorphism.

\begin{lemma}[Prime lifting]\label{lem:local-prime}
Let $(E,\sigma)$ be a difference ring, take $a\in E$ and put $b=\sigma(a)$. Suppose the ordinary ring homomorphism
\begin{equation}\label{eq:lambda}
 \lambda:E[a^{-1}]\longrightarrow E[b^{-1}],\qquad
  \frac{x}{a^r}\longmapsto\frac{\sigma(x)}{b^r}
\end{equation}
is flat. If $Q$ is a prime difference ideal of $E$ with $b\notin Q$, there is a prime difference ideal $Q_1$ such that
\begin{equation}\label{eq:lift-properties}
 Q_1\subseteq Q,\qquad b\notin Q_1,\qquad \sigma^{-1}(Q_1)=Q.
\end{equation}
\end{lemma}
\begin{proof}
Put $E_0=E[a^{-1}]$ and $E_1=E[b^{-1}]$. Since $b\notin Q$ and $Q$ is stable, both $a\notin Q$ and $a\notin\sigma^{-1}(Q)$. Thus
\[
 \mathfrak p=QE_0\subseteq\mathfrak p'=\sigma^{-1}(Q)E_0
\]
are prime ideals of $E_0$, and $\mathfrak q=QE_1$ is a prime of $E_1$. For every fraction $x/a^r$,
\[
 \lambda(x/a^r)\in\mathfrak q
 \iff \sigma(x)\in Q
 \iff x\in\sigma^{-1}(Q).
\]
Therefore $\lambda^{-1}(\mathfrak q)=\mathfrak p'$. Flat going-down gives a prime $\mathfrak q_1\subseteq\mathfrak q$ in $E_1$ with $\lambda^{-1}(\mathfrak q_1)=\mathfrak p$. Contract $\mathfrak q_1$ to an ordinary prime $Q_1$ of $E$. Localization shows $Q_1\subseteq Q$ and $b\notin Q_1$. Furthermore,
\[
 x\in Q
 \iff x/1\in\mathfrak p
 \iff \sigma(x)/1\in\mathfrak q_1
 \iff \sigma(x)\in Q_1.
\]
Hence $\sigma^{-1}(Q_1)=Q$. Finally $Q_1\subseteq Q=\sigma^{-1}(Q_1)$ implies $\sigma(Q_1)\subseteq Q_1$.
\end{proof}

\begin{lemma}[A reflexive prime below a prime]\label{lem:reflexive-prime}
Under the hypotheses of Lemma~\ref{lem:local-prime}, every prime difference ideal $Q$ avoiding $b$ contains a reflexive prime difference ideal.
\end{lemma}
\begin{proof}
Starting with $Q_0=Q$, apply Lemma~\ref{lem:local-prime} successively to construct
\[
 Q_0\supseteq Q_1\supseteq Q_2\supseteq\cdots,
 \qquad\sigma^{-1}(Q_{j+1})=Q_j,
 \qquad b\notin Q_j.
\]
The intersection $P=\bigcap_{j\ge0}Q_j$ is a proper prime ideal. Indeed, if $xy\in P$ and $x\notin P$, there is $k$ with $x\notin Q_k$. For all $j\ge k$, $x\notin Q_j$, hence $y\in Q_j$; the earlier ideals also contain $y$. Stability under $\sigma$ is inherited from the $Q_j$. Conversely, if $\sigma(x)\in P$, then $\sigma(x)\in Q_{j+1}$ for every $j$, so $x\in Q_j$ for every $j$. Thus $P$ is reflexive, and $P\subseteq Q$.
\end{proof}

\begin{theorem}[Localized flat-lifting identity]\label{thm:flat-identity}
Under the hypotheses of Lemma~\ref{lem:local-prime},
\begin{equation}\label{eq:flat-identity}
 \CC_E(0)=\PP_E(0)\cap\CC_E(b).
\end{equation}
\end{theorem}
\begin{proof}
The left-hand side is contained in the right-hand side. Conversely, let $Q$ be any prime difference ideal of $E$. If $b\in Q$, then $\CC_E(b)\subseteq Q$. If $b\notin Q$, Lemma~\ref{lem:reflexive-prime} produces a reflexive prime difference ideal $P\subseteq Q$. Lemma~\ref{lem:perfect-criterion} gives $\PP_E(0)\subseteq P\subseteq Q$. Thus the right-hand side of \eqref{eq:flat-identity} is contained in every prime difference ideal of $E$, and Lemma~\ref{lem:prime-representation} proves the result.
\end{proof}

\begin{remark}
No Noetherianity or injectivity of the difference endomorphism on $E$ is required. The map \eqref{eq:lambda} is not automatically flat: its flatness must be established in the application. The statement also permits zero localizations, since the case of a prime avoiding $b$ simply cannot occur if $b$ is nilpotent.
\end{remark}

\section{A uniform inverse-transform bound over inversive fields}\label{sec:uniform}

Throughout this section, we assume that $K$ is inversive.
Write $y_{i,j}=\sigma^j(y_i)$ and put
$R=K\{y_1,\ldots,y_n\}=K[y_{i,j}:1\leq i\leq n,\ j\geq 0]$. For $h\geq 1$, set
\begin{align*}
 A_h&=K[y_{i,j}:1\leq i\leq n,\ 0\leq j\leq h],\\
 A_h^-&=K[y_{i,j}:1\leq i\leq n,\ 0\leq j<h],\\
 B_h&=K[y_{i,j}:1\leq i\leq n,\ 1\leq j\leq h],\\
 R^+&=K[y_{i,j}:1\leq i\leq n,\ j\geq 1].
\end{align*}
Since $K$ is inversive, $\sigma$ induces ring isomorphisms
\(
 \sigma:A_h^-\xrightarrow{\sim}B_h\), and 
 \( 
 \sigma:R\xrightarrow{\sim}R^+.
\)

\begin{definition}
An ordinary ideal $\mathfrak p\subseteq A_h$ is
\emph{inverse-compatible} if
\[
 \sigma^{-1}(\mathfrak p\cap B_h)
 \subseteq \mathfrak p\cap A_h^-.
\]
Here $\sigma^{-1}$ denotes the inverse of
$\sigma:A_h^-\xrightarrow{\sim}B_h$, which sends $y_{i,j}$ to $y_{i,j-1}$ and acts on
coefficients by $\sigma_K^{-1}$.  
\end{definition}

The next lemma gives a tensor-product presentation of
$R/[\mathfrak p]_R$ and a corresponding factorization
of its difference endomorphism.

\begin{lemma}[Tensor presentation]\label{lem:tensor}
Let $\mathfrak p\subseteq A_h$ be an inverse-compatible ideal,
and set $ \mathfrak q=\mathfrak p\cap B_h$, $
 B=B_h/\mathfrak q,$ and $ D=A_h/\mathfrak p.$
Define
\[
 T=(\sigma^j(f):f\in\mathfrak p,\ j\geq 1)_{R^+},
 \qquad
 S=R^+/T,\qquad
 E=R/[\mathfrak p]_R.
\]
The natural maps from $B_h$ to $S$ and $D$ factor through
$B$, and the multiplication induces a canonical ring
isomorphism
\begin{equation}\label{eq:tensor}
 S\otimes_B D\xrightarrow{\sim}E,
 \qquad
 \overline{u}\otimes\overline{v}
 \longmapsto\overline{uv},
\end{equation}
where $u\in R^+$, $v\in A_h$.
Moreover, the difference
endomorphism $\sigma_E$ of $E$ is a composite
\begin{equation}\label{eq:sigma-factor}
 E\xrightarrow{\bar{\sigma}}S\xrightarrow{\iota}E,
\end{equation}
where $\bar{\sigma}:E\xrightarrow{\sim}S$ is induced by  $\sigma$, and  $\iota$  corresponds under \eqref{eq:tensor}  to the map
$s\mapsto s\otimes 1$. 
\end{lemma}

\begin{proof}
We first verify that $S$ is a $B$-algebra.
Indeed, for $g\in\mathfrak q$, inverse compatibility gives
\(
 f:=\sigma^{-1}(g)\in\mathfrak p\cap A_h^-,
\)
implying $g=\sigma(f)\in T$, proving $\mathfrak q\subseteq T$.
Hence $B_h\to S$ factors through $B$. The map $B_h\to D$
also factors through $B$, for  $ \mathfrak q=\mathfrak p\cap B_h$.

Note that \(A_h=B_h[y_{1,0},\ldots,y_{n,0}]\), 
\( R=R^+[y_{1,0},\ldots,y_{n,0}],
\) and  \([\mathfrak p]_R=\mathfrak pR+TR\).
Consequently, multiplication gives a canonical
isomorphism
\(
 R^+\otimes_{B_h}A_h\xrightarrow{\sim}R.
\)
Taking the corresponding quotients in the tensor
product, we obtain
\[  S\otimes_B D
 \cong (R^+/T)\otimes_{B_h}(A_h/\mathfrak p)\cong R/(TR+\mathfrak pR)
 =E.
\]
These isomorphisms send
$\overline{u}\otimes\overline{v}$ to $\overline{uv}$,
which proves \eqref{eq:tensor}.

Finally, under the ring isomorphism
$\sigma:R\xrightarrow{\sim}R^+$, the ideal
$[\mathfrak p]_R$ is sent to $T$.  Thus $\sigma$ induces
\[
 \bar{\sigma}:E\xrightarrow{\sim}S,\qquad
 f+[\mathfrak p]_R\longmapsto\sigma(f)+T.
\]
Since $T\subseteq[\mathfrak p]_R$, the inclusion
$R^+\subseteq R$ induces a map $\iota:S\to E$.
For every $f\in R$,
\[
 \iota\bigl(\theta(f+[\mathfrak p]_R)\bigr)
 =\sigma(f)+[\mathfrak p]_R
 =\sigma_E(f+[\mathfrak p]_R).
\]
This proves \eqref{eq:sigma-factor}. The description
of $\iota$ as $s\mapsto s\otimes 1$ follows from
\eqref{eq:tensor}.
\end{proof}

\begin{lemma}[Boundary equation]\label{lem:boundary}
If $\mathfrak p\subseteq A_h$ is prime and
inverse-compatible, there exists
$b\in B_h\setminus\mathfrak p$ such that
\begin{equation}\label{eq:boundary}
 \CC_R(\mathfrak p)
 =\PP_R(\mathfrak p)\cap\CC_R(\mathfrak p,b).
\end{equation}
\end{lemma}

\begin{proof}
Follow the notation of Lemma~\ref{lem:tensor}.
We first choose $b$ so that the map $\iota:S\to E$
becomes flat after localization, and then use the
factorization $\sigma_E=\iota\circ\bar{\sigma}$ to obtain
the flat localized shift map required by
Theorem~\ref{thm:flat-identity}.

Since $\mathfrak p$ is prime,  
$\mathfrak q=\mathfrak p\cap B_h$ is prime. Hence $B$
is a Noetherian domain. Moreover, 
$B\hookrightarrow  D$ is injective, and $D$ is a finitely generated
$B$-algebra, generated by the images of
$y_{1,0},\ldots,y_{n,0}$.
Applying generic freeness
\cite[Theorem~14.4]{Eisenbud1995} to the $B$-algebra
$D$ and the $D$-module $D$, we obtain a nonzero
$\beta\in B$ such that
$D[\beta^{-1}]$ is a free $B[\beta^{-1}]\text{-module}$.

Choose a lift $b\in B_h$ of $\beta$. Since
$\beta\neq 0$,  $b\notin\mathfrak q$, and
thus $b\notin\mathfrak p$. Denote the images of
$b$ in $S$ and $E$ by $b_S$ and $b_E$, respectively.
Localizing \eqref{eq:tensor} gives
\[
 E[b_E^{-1}]
 \cong
 S[b_S^{-1}]
 \otimes_{B[\beta^{-1}]}
 D[\beta^{-1}].
\]
Under this identification, the localized map
\(
 \iota_b:S[b_S^{-1}]
 \longrightarrow E[b_E^{-1}]
\)
is the canonical map into the tensor product.
Since extension of scalars preserves free modules,
$E[b_E^{-1}]$ is free, and hence flat, as an
$S[b_S^{-1}]$-module. Thus $\iota_b$ is flat.

Now set $a=\sigma^{-1}(b)\in A_h^-,$
and let $a_E$ denote its image in $E$. Since
$\sigma(a)=b$, we have
$\bar{\sigma}(a_E)=b_S,$ and $\sigma_E(a_E)=b_E.$
Localizing  $\bar{\sigma}$ gives an isomorphism
\[
 \bar{\sigma}_a:E[a_E^{-1}]
 \xrightarrow{\sim}S[b_S^{-1}].
\]
It follows that the composite
\[
 E[a_E^{-1}]
 \xrightarrow{\bar{\sigma}_a}
 S[b_S^{-1}]
 \xrightarrow{\iota_b}
 E[b_E^{-1}]
\]
is flat by sinding \(\frac{x}{a_E^r}
 \longmapsto
 \frac{\sigma_E(x)}{b_E^r}.\) 
It is therefore the localized shift map appearing
in Lemma~\ref{lem:local-prime}. Applying
Theorem~\ref{thm:flat-identity} to $E$ yields
\[
 \CC_E(0)=\PP_E(0)\cap\CC_E(b_E).
\]

Let $\pi:R\twoheadrightarrow E$ be the canonical difference
quotient map. By \eqref{eq:quotient},
\[
 \pi^{-1}\bigl(\CC_E(0)\bigr)
=\CC_R(\mathfrak p),
 \pi^{-1}\bigl(\PP_E(0)\bigr) =\PP_R(\mathfrak p),\,\,
 \pi^{-1}\bigl(\CC_E(b_E)\bigr)
 =\CC_R(\mathfrak p,b).
\]
Taking inverse images of the preceding identity
proves \eqref{eq:boundary}.
\end{proof}

\begin{theorem}[Uniform bound, inversive case]\label{thm:uniform-inv}
For every finite $F\subseteq R$ there is $N\ge0$ such that
\begin{equation}\label{eq:uniform-inv}
 \sigma^N\bigl(\PP_R(F)\bigr)\subseteq\CC_R(F),
 \qquad\text{equivalently }\quad
 \PP_R(F)=\sigma^{-N}\bigl(\CC_R(F)\bigr).
\end{equation}
\end{theorem}
\begin{proof}
Choose $h\ge1$ with $F\subseteq A_h$. For an \emph{arbitrary ordinary ideal} $\mathfrak a\subseteq A_h$ consider the property
\[
 \mathscr U(\mathfrak a):\quad
 \text{there is $N\ge0$ such that }
 \sigma^N\bigl(\PP_R(\mathfrak a)\bigr)\subseteq\CC_R(\mathfrak a).
\]
We prove $\mathscr U(\mathfrak a)$ for every $\mathfrak a$ by Noetherian induction in this \emph{fixed} polynomial ring $A_h$. Suppose the assertion fails, and choose a maximal counterexample $\mathfrak p$ with respect to inclusion. The unit ideal satisfies $\mathscr U$ and so $\mathfrak p$ is proper.

First, $\mathfrak p$ is prime. Otherwise choose $u,v\notin\mathfrak p$ with $uv\in\mathfrak p$ and set $\mathfrak a=\mathfrak p+(u)$, $\mathfrak b=\mathfrak p+(v)$. Both strictly contain $\mathfrak p$. Since $\mathfrak p^2\subseteq\mathfrak a\mathfrak b\subseteq\mathfrak p$, radicality and Lemma~\ref{lem:product-intersection} imply, for $\mathcal L=\CC_R$ and for $\mathcal L=\PP_R$,
\[
 \mathcal L(\mathfrak p)=\mathcal L(\mathfrak a)\cap\mathcal L(\mathfrak b).
\]
Take uniform exponents for $\mathfrak a$ and $\mathfrak b$ and let $N$ be their maximum. Stability under further shifts gives
$\sigma^N(\PP_R(\mathfrak p)\subseteq\CC_R(\mathfrak a)\cap\CC_R(\mathfrak b)=\CC_R(\mathfrak p)$, a contradiction.

Second, $\mathfrak p$ is inverse-compatible. If $g\in\mathfrak p\cap B_h$ but $a=\sigma^{-1}(g)\notin\mathfrak p$, the ordinary ideal $\mathfrak p'=\mathfrak p+(a)$ strictly contains $\mathfrak p$. Since $\sigma(a)=g\in\mathfrak p$, reflexivity gives
\[
 \PP_R(\mathfrak p')=\PP_R(\mathfrak p).
\]
The ideal $\sigma^{-1}(\CC_R(\mathfrak p))$ is radical well-mixed and contains $\mathfrak p'$; hence
\[
 \CC_R(\mathfrak p')\subseteq\sigma^{-1}(\CC_R(\mathfrak p)).
\]
A uniform exponent $M$ for $\mathfrak p'$ therefore yields the exponent $M+1$ for $\mathfrak p$, again a contradiction.

Finally, Lemma~\ref{lem:boundary} supplies $b\in B_h\setminus\mathfrak p$ with
\[
 \CC_R(\mathfrak p)=\PP_R(\mathfrak p)\cap\CC_R(\mathfrak p,b).
\]
The ideal $\mathfrak p+(b)$ strictly contains $\mathfrak p$ \emph{inside the same ring $A_h$}. By maximality, for some $M\ge0$,
\[
 \sigma^M\PP_R(\mathfrak p,b)\subseteq\CC_R(\mathfrak p,b).
\]
For $x\in\PP_R(\mathfrak p)$, this puts $\sigma^M(x)$ in $\CC_R(\mathfrak p,b)$, while difference stability puts it in $\PP_R(\mathfrak p)$. Their intersection equals $\CC_R(\mathfrak p)$, the final contradiction.

Apply the result to the ordinary ideal $\mathfrak a=(F)_{A_h}$. The closures of $\mathfrak a$ in $R$ equal those of $F$.  This proves the inclusion in \eqref{eq:uniform-inv};
the reverse inclusion follows from
Lemma~\ref{lem:reflexive-closure}, giving the asserted equality.
\end{proof}

\begin{remark}\label{rem:fixed-truncation}
The only ascending chain condition invoked in the proof is ordinary Noetherianity of $A_h$. Both operations on a maximal counterexample adjoin polynomials already in $A_h$: an inverse-shifted overlap equation in $A_h^-$, or a boundary equation in $B_h$. This is what makes the exponent uniform for a fixed finite system.

In Section~6, we remove the inversiveness assumption on $K$ by descent from its inversive closure. We then use the
Ritt--Raudenbush basis theorem to extend the result to arbitrary subsets $F\subseteq R$.
\end{remark}

\section{Finite bases and finite-type extensions}\label{sec:transfer}

Write $\operatorname{HB}(A)$ for the assertion that every radical well-mixed difference ideal of $A$ has a finite radical well-mixed basis. Here, no injectivity hypothesis is imposed on the endomorphism of $A$.

\begin{lemma}[Maximal counterexample]\label{lem:maxcounter}
If $\operatorname{HB}(A)$ fails, there is a radical well-mixed difference ideal $I\subseteq A$ maximal among those without a finite radical well-mixed basis. Every such $I$ is an ordinary prime ideal.
\end{lemma}
\begin{proof}
Chains of radical well-mixed ideals without finite bases have unions of the same kind. Indeed, if the union had a finite basis $F$, all elements of $F$ would lie in one member of the chain, forcing that member to equal the union. Zorn's lemma therefore gives a maximal counterexample $I$. Since the unit ideal is finitely based, $I\ne A$.

If $uv\in I$ with $u,v\notin I$, then $I_u=\CC_A(I,u)$ and $I_v=\CC_A(I,v)$ strictly contain $I$ and hence have finite bases, say $I_u=\CC_A(G)$ and $I_v=\CC_A(H)$ with $G,H$ finite. Every product of an element of $I+(u)$ and an element of $I+(v)$ belongs to $I$. Lemma~\ref{lem:product-intersection} shows $I_u\cap I_v=I$ and also
\[
 I=\CC_A(G)\cap\CC_A(H)
   =\CC_A(\{gh:g\in G,h\in H\}),
\]
a contradiction. Thus $I$ is prime.
\end{proof}

\begin{lemma}[Difference quotients]\label{lem:quotient-HB}
Suppose $J=\CC_A(F)$ for a finite set $F$, and $I\supseteq J$ is radical well-mixed. Then $I$ has a finite basis if and only if $I/J$ has a finite basis in $A/J$. A maximal counterexample $I$ remains maximal in $A/J$. Moreover, $\operatorname{HB}(A)$ implies
$\operatorname{HB}(A/L)$ for every difference ideal $L\subseteq A$, without any finite-basis assumption on $L$.
\end{lemma}
\begin{proof}
A finite basis in $A$ maps to one in $A/J$ by Lemma~\ref{lem:functorial}. Conversely, if $I/J$ is generated under radical well-mixed closure by images of a finite $G\subseteq I$, then
\[
 I=\CC_A(G\cup J)=\CC_A(G\cup F).
\]
This also proves preservation of maximality.
Let
$\rho:A\to A/L$ be  a general difference quotient,  and assume $\operatorname{HB}(A)$.
For a radical well-mixed difference ideal $H\subseteq A/L$, 
its inverse image $\rho^{-1}(H)$ is radical and well-mixed. Choose a finite set $G\subseteq A$ such that $\rho^{-1}(H)=\CC_A(G)$. Passing to the quotient gives
\(
 H=\CC_{A/L}\bigl(\rho(G)\bigr).
\)
Hence $\operatorname{HB}(A/L)$ holds. 
\end{proof}

We next record a generic finite-generation statement for ordinary ideals over a possibly non-Noetherian base.

\begin{lemma}[Generic finite generation]\label{lem:generic-fg}
Let $B$ be an integral domain, $A$ a finitely generated ordinary $B$-algebra, and $L\subseteq A$ an ordinary ideal. There are $0\ne e\in B$ and a finite set $G\subseteq L$ such that
\begin{equation}\label{eq:generic-fg}
 L A[e^{-1}]=(G)_{A[e^{-1}]}.
\end{equation}
Equivalently, for each $x\in L$ there is $m=m(x)\ge0$ with $e^m x\in(G)_A$.
\end{lemma}
\begin{proof}
Choose a surjective $B$-algebra map $\rho:B[X_1,\ldots,X_t]\twoheadrightarrow A$, and put $M=\rho^{-1}(L)$. Let $Q=\Frac(B)$ and fix a monomial order. The ideal $MQ[X]$ has a finite monic Gr\"obner basis $g_1,\ldots,g_s$ over $Q$. Each $g_i$ is a finite $Q[X]$-linear combination of elements of $M$. Clearing denominators gives $d_i\in B\setminus\{0\}$ such that $h_i=d_i g_i\in M$. Set $e=\prod_i d_i$, taking $e=1$ if $MQ[X]$ is zero.

Over $B[e^{-1}]$, all $g_i$ are monic and lie in the ideal generated by the $h_i$. For any $f\in MB[e^{-1}][X]$, monic division yields $f=\sum q_i g_i+r$, where no term of $r$ is divisible by a leading monomial of the Gr\"obner basis. After embedding in $Q[X]$, the remainder lies in $MQ[X]$, so $r=0$. Consequently
\[
 MB[e^{-1}][X]=(h_1,\ldots,h_s)_{B[e^{-1}][X]}.
\]
Take images under $\rho$ to obtain \eqref{eq:generic-fg}. Clearing denominators gives the stated saturation formulation; it also works if $A\to A[e^{-1}]$ is not injective.
\end{proof}

\begin{lemma}[Annihilators]\label{lem:ann}
Let $(A,\sigma)$ be a reduced well-mixed difference ring. For every $x\in A$ the ordinary ideal $\Ann_A(x)$ is radical well-mixed. In particular,
\begin{equation}\label{eq:ann-criterion}
 ex=0\quad\text{and}\quad x\in\CC_A(e)
 \quad\Longrightarrow\quad x=0.
\end{equation}
\end{lemma}
\begin{proof}
If $rx=0$, well-mixedness gives $x\sigma(r)=0$, so the annihilator is a difference ideal. If $rsx=0$, well-mixedness applied to $(xr)s=0$ gives $xr\sigma(s)=0$. If $r^m x=0$, then $(rx)^m=r^m x^m=0$, and reducedness yields $rx=0$. Hence the annihilator is radical well-mixed. If $ex=0$, it contains $e$ and therefore $\CC_A(e)$. Under the further hypothesis $x\in\CC_A(e)$, we have $x^2=0$, and reducedness gives $x=0$.
\end{proof}

\begin{theorem}[Finite-type transfer]\label{thm:finite-type}
Let $f:(B,\sigma_B)\to(A,\sigma_A)$ be a difference-ring homomorphism such that $A$ is finitely generated as an ordinary $B$-algebra via $f$. If $\operatorname{HB}(B)$ holds, then $\operatorname{HB}(A)$ holds.
\end{theorem}
\begin{proof}
First suppose that $B\subseteq A$. If $\operatorname{HB}(A)$ fails, take a maximal counterexample $I$ as in Lemma~\ref{lem:maxcounter}; in particular $I$ is prime. Its contraction $\mathfrak q=I\cap B$ is a prime radical well-mixed ideal. Since $\operatorname{HB}(B)$ holds, choose a finite $F\subseteq\mathfrak q$ such that $\mathfrak q=\CC_B(F)$. Put $J=\CC_A(F)$. Then $J\subseteq I$ and
\begin{equation}\label{eq:contract-J}
 J\cap B=\mathfrak q.
\end{equation}
Indeed, the first inclusion follows from $J\subseteq I$, and the reverse follows because $J\cap B$ is radical well-mixed and contains $F$.

Replace $B,A,I$ by $B/\mathfrak q,A/J,I/J$. Equation~\eqref{eq:contract-J} ensures that the induced map of the first two quotient rings is injective. Lemma~\ref{lem:quotient-HB} preserves both the base finite-basis property and maximality of the counterexample. We may therefore assume
\begin{equation}\label{eq:domain-reduction}
 B\text{ is a domain},\qquad I\cap B=0.
\end{equation}

Apply Lemma~\ref{lem:generic-fg} to the ordinary ideal $I\subseteq A$. We obtain $0\ne e\in B$ and a finite $G\subseteq I$ such that for each $x\in I$ some $m\ge1$ satisfies $e^m x\in(G)_A$. Then
\[
 (ex)^m=x^{m-1}(e^m x)\in(G)_A\subseteq\CC_A(G),
\]
and radicality gives
\begin{equation}\label{eq:eI}
 eI\subseteq\CC_A(G).
\end{equation}
By \eqref{eq:domain-reduction}, $e\notin I$. The radical well-mixed ideal $\CC_A(I,e)$ strictly contains $I$ and therefore has a finite basis by maximality. Finitarity (Lemma~\ref{lem:finitarity}) supplies a finite $H\subseteq I$ such that
\begin{equation}\label{eq:extract-H}
 \CC_A(I,e)=\CC_A(H,e).
\end{equation}
To see this, express every member of a finite basis of the left-hand side as a consequence of a finite subset of $I\cup\{e\}$ and collect the elements coming from $I$.

Put $J_1=\CC_A(G\cup H)$ and pass to $\overline A=A/J_1$, which is reduced and well-mixed. For $x\in I$, write $\bar e,\bar x$ for the images in $\overline A$. Equation~\eqref{eq:eI} gives $\bar e\bar x=0$; equation~\eqref{eq:extract-H}, after passing to the quotient, gives $\bar x\in\CC_{\overline A}(\bar e)$. Lemma~\ref{lem:ann} yields $\bar x=0$. Thus $I\subseteq J_1$, and the reverse inclusion follows from $G\cup H\subseteq I$. We have expressed $I=\CC_A(G\cup H)$, a contradiction.
Thus, we proved $\operatorname{HB}(B)$ for $B\subseteq A$. 

For a general $f$, the isomorphism of difference rings $ f(B)\cong B/\ker(f)$
and Lemma~\ref{lem:quotient-HB} imply that
$\operatorname{HB}(f(B))$ holds.
Since $A$ is finitely generated as an ordinary
$f(B)$-algebra, applying the inclusion case to
$f(B)\subseteq A$ gives $\operatorname{HB}(A)$.
\end{proof}

\begin{remark}
The only ordinary finite-generation input is Lemma~\ref{lem:generic-fg}; the base $B$ itself need not be Noetherian. The proof does not assume that an arbitrary difference quotient of $A$ is Noetherian as an ordinary ring.
\end{remark}

\section{The basis theorem over inversive fields}\label{sec:inversive-main}

We continue to assume that $(K,\sigma_K)$ is inversive and set $R=K\{y_1,\ldots,y_n\}$.

\begin{theorem}[Inversive-field case]\label{thm:inversive-basis}
Every radical well-mixed difference ideal in $R$ has a finite radical well-mixed basis.
\end{theorem}
\begin{proof}
Suppose otherwise, and choose a maximal counterexample $I$ by Lemma~\ref{lem:maxcounter}. Put $P=\PP_R(I)$. By Lemma~\ref{lem:reflexive-closure},
\begin{equation}\label{eq:P-union}
 P=\bigcup_{r\ge0}\sigma^{-r}(I).
\end{equation}
By Theorem~\ref{thm:Ritt--Raudenbush}, choose finite $E\subseteq P$ with $P=\PP_R(E)$. Each $e\in E$ satisfies $\sigma^{r(e)}(e)\in I$ for some $r(e)\ge0$ by \eqref{eq:P-union}. Choosing $r\ge r(e)$ for all $e$ and putting $F=\sigma^r(E)$, we have $F\subseteq I$ and
\begin{equation}\label{eq:PF}
 \PP_R(F)=P.
\end{equation}
The equality holds because $P$ contains $F$, while the perfect ideal $\PP_R(F)$ is reflexive and therefore contains each $e\in E$. The case $E=\varnothing$ is covered by taking $r=0$.

Let $J=\CC_R(F)\subseteq I$. Theorem~\ref{thm:uniform-inv} and \eqref{eq:PF} furnish $N\ge0$ such that
\begin{equation}\label{eq:key-PJ}
 P=\sigma^{-N}(J).
\end{equation}
Put $A=R/J$, write $\tau$ for its induced endomorphism, and let $D=\tau^N(A)\subseteq A$. Since $\tau(D)\subseteq D$, this is a difference subring. Formula~\eqref{eq:key-PJ} says that the kernel of $\tau^N:A\to A$ is $P/J$. The first isomorphism theorem therefore gives a difference-ring isomorphism
\begin{equation}\label{eq:D-isomorphism}
 \tau^N: \,R/P\xrightarrow{\sim}D,\qquad
 f+P\longmapsto\sigma^N(f)+J.
\end{equation}

We claim that $\operatorname{HB}(R/P)$ holds. The zero ideal of $R/P$ is radical well-mixed because $P$ is perfect, and its empty basis is finite. If $L$ is a nonzero radical well-mixed ideal of $R/P$, its inverse image in $R$ strictly contains $P\supseteq I$. By maximality of $I$, the inverse image has a finite radical well-mixed basis; projecting that basis to $R/P$ gives one for $L$. Thus $\operatorname{HB}(R/P)$ holds, and \eqref{eq:D-isomorphism} implies $\operatorname{HB}(D)$.

It remains to show that $A$ is a finite-type ordinary $D$-algebra. Let $\bar y_{i,j}$ be the class of $y_{i,j}$ in $A$. If $j\ge N$, then
\[
 \bar y_{i,j}=\tau^N(\bar y_{i,j-N})\in D.
\]
Moreover, since $\sigma_K$ is an automorphism, every $c\in K$ has $c'\in K$ with $\sigma_K^N(c')=c$; hence $c+J=\tau^N(c'+J)\in D$. Consequently
\begin{equation}\label{eq:finite-type-over-D}
 A=D[\bar y_{i,j}:1\le i\le n,\ 0\le j<N].
\end{equation}
The set of displayed generators is finite. Theorem~\ref{thm:finite-type} yields $\operatorname{HB}(A)$, so $I/J$ has a finite basis. Lemma~\ref{lem:quotient-HB}, together with $J=\CC_R(F)$ for finite $F$, lifts it to a finite basis of $I$, contrary to choice.
\end{proof}

\begin{corollary}\label{cor:inversive-algebras}
Every finitely difference-generated difference algebra over an inversive difference field has the finite radical well-mixed basis property and the ascending chain condition for radical well-mixed ideals.
\end{corollary}
\begin{proof}
Such an algebra is a difference quotient of a difference polynomial ring in finitely many indeterminates. Apply Theorem~\ref{thm:inversive-basis} and quotient preservation in Lemma~\ref{lem:quotient-HB}. If $I_0\subseteq I_1\subseteq\cdots$ is an ascending chain, its union is radical well-mixed. A finite basis of the union lies inside some $I_N$, forcing stabilization.
\end{proof}

\section{Descent from the inversive closure}\label{sec:descent}

We now remove the coefficient-field hypothesis from the preceding two theorems. The argument is not ordinary base change of ideals: extending a radical well-mixed ideal to a larger difference field need not leave it radical or well-mixed. What is needed, and what we prove, is exact \emph{contraction after taking radical well-mixed closure}.

Let $(K,\sigma_K)$ be an arbitrary difference field and let $K^{\inv}$ denote its inversive closure. Concretely, inside $K^{\inv}$, put
\[
 K_m=\sigma^{-m}(K)\qquad(m\ge0),
\]
where $\sigma$ on $K^{\inv}$ is an automorphism. Then by \cite{Chatzidakis},
\begin{equation}\label{eq:K-tower}
 K=K_0\subseteq K_1\subseteq\cdots,\qquad
 K^{\inv}=\bigcup_{m\ge0}K_m,\qquad
 \sigma(K_m)=K_{m-1}\quad(m\ge1).
\end{equation}
In particular, each $K_m/K_{m-1}$ is a (possibly infinitely generated) field extension, and hence flat. 

Let $A$ be any difference $K$-algebra. Set
\[
 A_m=K_m\otimes_K A,\qquad
 A^{\inv}=K^{\inv}\otimes_K A=\varinjlim_m A_m.
\]
Since
\(
 A_m\cong K_m\otimes_{K_{m-1}}A_{m-1},
\)
each transition map $A_{m-1}\to A_m$ is a base change
of the faithfully flat field extension
$K_{m-1}\subseteq K_m$. Hence these maps are faithfully
flat and, in particular, injective. We therefore regard
the $A_m$ as subrings of $A^{\inv}$.
Each $A_m$ is a difference ring via
\[
 \sigma_m(c\otimes a)=\sigma(c)\otimes\sigma_A(a).
\]
For $m\ge1$, the map $\sigma_m$ actually has image in $A_{m-1}$, because $\sigma(K_m)=K_{m-1}$ and $\sigma_A(A)\subseteq A$. Its restriction to $A_{m-1}$ agrees with that ring's difference endomorphism.

\begin{proposition}[Prime difference lying-over]\label{prop:prime-descent}
For every prime difference ideal $Q_0$ of $A$, there is a prime difference ideal $Q$ of $A^{\inv}$ with
\[
 Q\cap A=Q_0.
\]
More precisely, $Q$ may be chosen as the union of prime difference ideals $Q_m\subseteq A_m$ satisfying
\begin{equation}\label{eq:compatible-primes}
 Q_m\cap A_{m-1}=Q_{m-1},\qquad
 \sigma(Q_m)\subseteq Q_{m-1}\quad(m\ge1).
\end{equation}
\end{proposition}
\begin{proof}
Suppose $Q_{m-1}$ has been constructed. Consider the inverse image
\[
 \mathfrak q'=\sigma_m^{-1}(Q_{m-1})\subseteq A_m.
\]
Because $Q_{m-1}$ is an ordinary prime ideal and $\sigma_m:A_m\to A_{m-1}$ is a ring homomorphism, $\mathfrak q'$ is prime. Let
\[
 \mathfrak p'=\mathfrak q'\cap A_{m-1}.
\]
The stability of $Q_{m-1}$ implies $Q_{m-1}\subseteq\mathfrak p'$: for $x\in Q_{m-1}$, $\sigma_m(x)=\sigma_{m-1}(x)\in Q_{m-1}$.

Since the inclusion $A_{m-1}\hookrightarrow A_m$ is flat, apply ordinary flat going-down \cite[Lemma 10.11]{Eisenbud1995} to the chain
\[
 Q_{m-1}\subseteq\mathfrak p'
\]
and to the prime $\mathfrak q'$ above $\mathfrak p'$. We obtain a prime ideal $Q_m\subseteq\mathfrak q'$ of $A_m$ contracting exactly to $Q_{m-1}$. As $Q_m\subseteq\sigma_m^{-1}(Q_{m-1})$, we have
\[
 \sigma(Q_m)\subseteq Q_{m-1}\subseteq Q_m.
\]
Thus $Q_m$ is a prime difference ideal, and both requirements in \eqref{eq:compatible-primes} hold.

Inductively construct the $Q_m$ for all $m$. Their union $Q=\bigcup_{m\ge0}Q_m$ inside $A^{\inv}$ is an ordinary prime ideal: if $xy\in Q$, the two factors and their product lie in some common $A_m$, where primality applies. It is a difference ideal by \eqref{eq:compatible-primes} and by the stability of $Q_0$. Its contraction to $A$ is $Q_0$ because every transition map contracts $Q_m$ to $Q_{m-1}$. This proves the proposition.
\end{proof}

\begin{theorem}[Exact contraction of radical well-mixed closure]\label{thm:contraction}
Let $A$ be any difference $K$-algebra and $A^{\inv}=K^{\inv}\otimes_K A$, with the induced difference structure. For every subset $U\subseteq A$,
\begin{equation}\label{eq:closure-contract}
 \CC_{A^{\inv}}(U)\cap A=\CC_A(U).
\end{equation}
\end{theorem}
\begin{proof}
The inverse image of $\CC_{A^{\inv}}(U)$ under $A\hookrightarrow A^{\inv}$ is radical well-mixed and contains $U$. Thus it contains $\CC_A(U)$. Conversely, let $Q_0$ be any prime difference ideal of $A$ containing $U$. Proposition~\ref{prop:prime-descent} lifts $Q_0$ to a prime difference ideal $Q$ of $A^{\inv}$, and $Q$ contains $U$. By Lemma \ref{lem:prime-representation},
\[
 \CC_{A^{\inv}}(U)\subseteq Q,
 \qquad\text{hence}\qquad
 \CC_{A^{\inv}}(U)\cap A\subseteq Q_0.
\]
Intersect over all such $Q_0$ and apply Lemma~\ref{lem:prime-representation}. If there are no such primes, both closures are the whole ring by the same lemma. This proves \eqref{eq:closure-contract}.
\end{proof}

\begin{example}[Ordinary extension need not preserve well-mixedness]\label{ex:noninversive-extension}
Let $k$ be a field of characteristic different from $2$, put $K=k(t)$, and let $\sigma_K$ fix $k$ and send $t$ to $t^2$. Define
\[
 I=(y_0^2-t,\ y_1-t,\ y_2-t^2,\ y_3-t^4,\ldots)
 \subseteq K\{y\}.
\]
This is a prime difference ideal: its quotient is the field $K(a)$, with $a^2=t$, on which $\sigma(a)=t$. Now pass to $K^{\inv}=\bigcup_{m\ge0}k(t^{1/2^m})$, and let $a=t^{1/2}\in K^{\inv}$. The ordinary extension of $I$ has quotient
\[
 K^{\inv}\{y\}/IK^{\inv}\{y\}
 \cong K^{\inv}[y_0]/(y_0^2-t)
 \cong K^{\inv}\times K^{\inv}.
\]
In this quotient the nonzero orthogonal idempotents
$e_\pm=\tfrac12(1\pm y_0/a)$ satisfy
$\sigma(e_+)=1$ and $\sigma(e_-)=0$. Since
$e_-e_+=0$ but $e_-\sigma(e_+)=e_-\ne0$, the extended ordinary ideal is not well-mixed. Thus one must take radical well-mixed closure \emph{before} contraction in Theorem~\ref{thm:contraction}.
\end{example}

\begin{corollary}[Finite-basis descent]\label{cor:HB-descent}
If $\operatorname{HB}(A^{\inv})$ holds, then $\operatorname{HB}(A)$ holds.
\end{corollary}
\begin{proof}
Let $I$ be a radical well-mixed difference ideal of $A$, and put $\widetilde I=\CC_{A^{\inv}}(I)$. By assumption, choose a finite $G\subseteq\widetilde I$ such that $\widetilde I=\CC_{A^{\inv}}(G)$. Finitarity gives, for every $g\in G$, a finite $F_g\subseteq I$ with $g\in\CC_{A^{\inv}}(F_g)$. Put $F=\bigcup_{g\in G}F_g$. Then
\[
 \widetilde I=\CC_{A^{\inv}}(F).
\]
Contract to $A$ and apply Theorem~\ref{thm:contraction} twice to obtain
\[
 I=\CC_{A^{\inv}}(I)\cap A
   =\CC_{A^{\inv}}(F)\cap A=\CC_A(F).
\]
Thus $I$ has the required finite basis.
\end{proof}

\begin{corollary}[Uniform-bound descent]\label{cor:uniform-descent}
Let $F\subseteq A$ be finite. Suppose an integer $N$ satisfies
\[
 \sigma^N\PP_{A^{\inv}}(F)\subseteq\CC_{A^{\inv}}(F).
\]
Then the same integer satisfies
\[
 \sigma^N\PP_A(F)\subseteq\CC_A(F).
\]
\end{corollary}
\begin{proof}
The natural map $A\to A^{\inv}$ is a difference-ring homomorphism. By Lemma~\ref{lem:functorial}, every $x\in\PP_A(F)$ belongs to $\PP_{A^{\inv}}(F)$. Hence $\sigma^N(x)\in\CC_{A^{\inv}}(F)\cap A=\CC_A(F)$ by Theorem~\ref{thm:contraction}.
\end{proof}

\begin{proof}[Proof of Theorems~\ref{thm:main} and~\ref{thm:uniform-intro}]
For $R=K\{\mathbf y\}$, coefficient extension gives
\[
 R^{\inv}=K^{\inv}\otimes_K R\cong K^{\inv}\{\mathbf y\}.
\]
The coefficient field $K^{\inv}$ is inversive. Therefore Theorem~\ref{thm:inversive-basis} gives $\operatorname{HB}(R^{\inv})$, and Corollary~\ref{cor:HB-descent} gives $\operatorname{HB}(R)$. This is the finite-basis statement of Theorem~\ref{thm:main}. For a finite $F\subseteq R$, Theorem~\ref{thm:uniform-inv} supplies $N$ with the uniform bound in $R^{\inv}$, and Corollary~\ref{cor:uniform-descent} transfers it to $R$. Reflexivity of $\PP_R(F)$ then gives the equality in \eqref{eq:uniform-main}, proving Theorem~\ref{thm:uniform-intro}.

Finally, if $I_0\subseteq I_1\subseteq\cdots$ is an ascending chain of radical well-mixed ideals of $R$, its union $I$ is radical well-mixed by Lemma~\ref{lem:finitarity}. Choose finite $F\subseteq I$ with $I=\CC_R(F)$. Some $I_N$ contains $F$; then $I=I_N$ and the chain stabilizes.
\end{proof}

The Ritt--Raudenbush basis theorem allows us to remove
the finiteness assumption in Theorem~\ref{thm:uniform-intro}.

\begin{corollary}\label{cor:uniform-arbitrary}
Let $(K,\sigma_K)$ be an arbitrary difference field and
let $R=K\{y_1,\ldots,y_n\}$.
For every subset $F\subseteq R$, there exists an integer
$N=N(F)\geq0$ such that
\[
 \PP_R(F)=\sigma^{-N}\!\bigl(\CC_R(F)\bigr).
\]
Equivalently, $\sigma^N(\PP_R(F))\subseteq\CC_R(F)$.
\end{corollary}

\begin{proof}
By Theorem \ref{thm:Ritt--Raudenbush}, the perfect difference ideal
$\PP_R(F)$ has a finite perfect basis.
Lemma~\ref{lem:finitarity} therefore yields a finite subset
$F_0\subseteq F$ such that
\[
 \PP_R(F_0)=\PP_R(F).
\]
Applying Theorem~\ref{thm:uniform-intro} to $F_0$, we obtain
an integer $N\geq0$ for which
\[
 \sigma^N\bigl(\PP_R(F)\bigr)
 =\sigma^N\bigl(\PP_R(F_0)\bigr)
 \subseteq\CC_R(F_0)
 \subseteq\CC_R(F).
\]
Thus $\PP_R(F)\subseteq\sigma^{-N}(\CC_R(F))$.
The reverse inclusion follows because
$\CC_R(F)\subseteq\PP_R(F)$ and $\PP_R(F)$ is reflexive.
\end{proof}

\begin{proof}[Proof of Corollary~\ref{cor:intro}]
Choose a finite set of difference generators $a_1,\ldots,a_n$ of $A$, that is, $A=K\{a_1,\ldots,a_n\}$. The evaluation homomorphism
\[
 \pi:K\{y_1,\ldots,y_n\}\twoheadrightarrow A,
 \qquad y_i\longmapsto a_i,
\]
is a surjective difference $K$-algebra homomorphism. For a radical well-mixed ideal $J\subseteq A$, its inverse image $I=\pi^{-1}(J)$ is radical well-mixed in $R$ by Lemma~\ref{lem:functorial}. Theorem~\ref{thm:main} gives finite $F\subseteq I$ with $I=\CC_R(F)$; passing to the quotient gives $J=\CC_A(\pi(F))$. For an ascending chain of ideals of $A$, use finitarity and the same finite-basis argument as above. Conversely, if the chain condition holds but a radical well-mixed ideal has no finite basis, adjoining one new element outside each previously finitely based closure yields an infinite strictly increasing chain. This establishes the equivalence.
\end{proof}

\begin{remark} \label{rem:two-going-down}
The flat-lifting argument of Section~\ref{sec:lifting} constructs a smaller prime difference ideal while controlling an inverse transform in a fixed truncated ring. Proposition~\ref{prop:prime-descent} uses flat going-down in a different way: it constructs compatible prime difference ideals over an increasing sequence of coefficient fields. The latter argument is independent of the uniform-bound proof and removes inversiveness without requiring an ordinary finite-type extension of coefficient fields.
\end{remark}

\section{Examples and final remarks}\label{sec:examples}

\begin{example} \label{ex:order}
The exponent cannot be chosen independently of the equations.
Indeed,
let $R=K\{y\}$, write $y_j=\sigma^j(y)$ and fix $r\ge1$. One has
\[
 \CC_R(y_r)=(y_r,y_{r+1},\ldots),\qquad
 \PP_R(y_r)=(y_0,y_1,\ldots).
\]
The first ideal is prime and difference-stable. The second is perfect, and reflexivity forces $y_0$ into every perfect ideal containing $y_r$. Thus
\[
 \PP_R(y_r)=\sigma^{-r}\bigl(\CC_R(y_r)\bigr).
\]
No exponent $N<r$ is sufficient: $y_0\in\PP_R(y_r)$, whereas $\sigma^N(y_0)=y_N\notin\CC_R(y_r)$.
\end{example}

The following example shows that perfect closure can also be strictly larger than the radical well-mixed closure. 
\begin{example} \label{ex:different-closures}
In $R=K\{x,y\}$ let $F=\{x\sigma(y)\}$. Well-mixedness, using the two factors in the reverse order, gives $\sigma(y)\sigma(x)=\sigma(xy)\in\CC_R(F)$. Reflexivity yields $xy\in\PP_R(F)$. On the other hand,
\[
 Q=(y_1,y_2,\ldots)\subseteq K[x_j,y_j:j\ge0]
\]
is a prime difference ideal containing $F$ and not containing $x_0y_0$. By Lemma~\ref{lem:prime-representation}, $xy\notin\CC_R(F)$.
\end{example}

\begin{remark} 
The ideal $(y_0,y_1,\ldots)$ of Example~\ref{ex:order} has the one-element radical well-mixed basis $\{y_0\}$ but is not finitely generated as an ordinary ideal. Thus the finite-basis theorem does not assert ordinary Noetherianity of a difference polynomial ring, nor ordinary finite generation of its radical well-mixed ideals.
\end{remark}

\begin{remark}
The results hold over an arbitrary difference field with a
single distinguished endomorphism. No assumptions of
characteristic zero, algebraic closedness, perfection,
or inversiveness are imposed. The induced endomorphism
on a finitely generated difference algebra need not be
injective. The present proof does not provide an effective
uniform exponent or address the case of several commuting
difference operators. Extending the prime-lifting and
fixed-truncation arguments to that setting, as well as
obtaining effective bounds, remains a direction for
future work.
\end{remark}

\section*{Acknowledgements}
The authors acknowledge the use of GPT-6 Astra to assist with brainstorming, mathematical development, and manuscript drafting.
The authors are solely responsible for the final content, analysis, and conclusions.

\section*{Funding}
WL was partially supported by the Strategic Priority Research Program of the Chinese Academy of Sciences under Grant XDA0480503, National Key R\&D Program of China 2023YFA1009401, NSFC Grants (12671568, 12288201). JW was supported by the National Key R\&D Program of China (No.~2023YFA1009401) and the National Natural Science Foundation of China (No.~12171324).

\end{document}